\documentclass[12pt, reqno]{amsart}
\usepackage{amsmath, amsthm, amscd, amsfonts, amssymb, graphicx, color}
\usepackage[bookmarksnumbered, colorlinks, plainpages]{hyperref}
\hypersetup{colorlinks=true,linkcolor=red, anchorcolor=green, citecolor=cyan, urlcolor=red, filecolor=magenta, pdftoolbar=true}

\input{mathrsfs.sty}

\newtheorem{theorem}{Theorem}[section]
\newtheorem{lemma}[theorem]{Lemma}
\newtheorem{proposition}[theorem]{Proposition}
\newtheorem{corollary}[theorem]{Corollary}
\theoremstyle{definition}

\theoremstyle{remark}
\newtheorem{remark}[theorem]{Remark}
\numberwithin{equation}{section}

\DeclareMathOperator\oo{O}
\begin{document}

\title[Inequalities for operator space numerical radius ]{Inequalities for operator space numerical radius of $2\times 2$ block matrices }

\author[M.S. Moslehian and M. Sattari]{Mohammad Sal Moslehian and Mostafa Sattari}

\address{Department of Pure Mathematics, Center Of Excellence in Analysis on Algebraic Structures (CEAAS), Ferdowsi University of Mashhad, P. O. Box 1159, Mashhad 91775, Iran}
\email{moslehian@um.ac.ir}
\email{msattari.b@gmail.com}



\begin{abstract}
In this paper, we study the relationship between operator space norm and operator space numerical radius on the matrix space $\mathcal{M}_n(X)$, when $X$ is a numerical radius operator space. Moreover, we establish several inequalities for operator space numerical radius and the maximal numerical radius norm of $2\times 2$ operator matrices and their off-diagonal parts. One of our main results states that if $(X, (O_n))$ is an operator space, then
\begin{align*}
\frac12\max\big(W_{\max}(x_1+x_2)&, W_{\max}(x_1-x_2) \big)\\ &\le W_{\max}\Big(\begin{bmatrix} 0 & x_1 \\ x_2 & 0 \end{bmatrix}\Big)\\
&\hspace{1.5cm}\le \frac12\left(W_{\max}(x_1+x_2)+ W_{\max}(x_1-x_2) \right)
\end{align*}
for all $x_1, x_2\in \mathcal{M}_n(X)$.
\end{abstract}

\maketitle

\section{Introduction}
Let $\mathcal{B}(H)$ denote the $C^*$-algebra of bounded linear operators acting on a Hilbert space $H$. Let $\|a\|_n$ denote the operator norm and $w_n(a)$ stand for the numerical radius norm of an element $a$ in the $n\times n$ matrix algebra $\mathcal{M}_n(\mathcal{B}(H))$ identifying with $\mathcal{B}(H^{(n)})$ in a natural way, where $H^{(n)}$ is the direct sum of $n$ copy of $H$. Recall that the numerical
radius norm of $a$ is given by
$w_n(a)= \sup \{ | \langle a x, x \rangle | : x \in H^{(n)} , \| x \| =1 \}.$
An (abstract) operator space is a complex linear space $X$ together with a sequence of norms $\oo_n(\cdot )$ $(n=1,2,\ldots)$ defined on the $n\times n$ matrix space $\mathcal{M}_n(X)$ satisfying the following Ruan's axioms (cf. \cite{ER}):

\begin{equation*}
\oo_{m+n}\bigg(\begin{bmatrix}
x & 0 \\ 0 &y
\end{bmatrix} \bigg)=\max \big\{ \oo_m(x),\oo_n(y)\big\},
\end{equation*}

\begin{equation*}
\oo_n(\alpha x\beta )\le \|\alpha \|\oo_m(x)\|\beta\|.
\end{equation*}
for all $x\in \mathcal{M}_m(X)$, $y\in \mathcal{M}_n(X)$, $\alpha \in \mathcal{M}_{n,m}(\mathbb{C})$ and $\beta \in \mathcal{M}_{m,n}(\mathbb{C})$.\\
Ruan \cite{Ruan} proved that if $(X, (\oo_n))$ is an operator space, then there is a complete isometry $\psi$ from $X$ to $\mathcal{B}(H)$ for some Hilbert space $H$ in the sense that $\oo_n(x)=\|\psi_n(x)\|_n$ for all $x\in \mathcal{M}_n(X)$ and $n\in \mathbb{N}$, where $\Vert\cdot\Vert_n$ is the usual operator norm of $\mathcal{M}_n(\mathcal{B}(H))$.

Itoh and Nagisa \cite{Itoh} introduced the notion of (abstract) numerical radius operator space (NROS), see also \cite{Itoh2}. By a numerical radius operator space we mean a complex linear space $X$ admitting a sequence of norms $W_n(\cdot)$ on $\mathcal{M}_n(X)$, $n\in \mathbb{N}$, for which
\begin{equation}\label{eq1.1}
W_{m+n}\bigg(\begin{bmatrix}
x &0\\ 0& y
\end{bmatrix}\bigg)=\max \big\{W_m(x),W_n(y)\big\},
\end{equation}
\begin{equation}\label{eq1.2}
W_n(\alpha x\alpha^*)\le \|\alpha\|^2W_m(x).
\end{equation}
for all $x\in \mathcal{M}_m(X)$, $y\in \mathcal{M}_n(X)$ and $\alpha\in \mathcal{M}_{n,m}(\mathbb{C})$, where $\alpha^*$ is the conjugate transpose of $\alpha$.

They also showed that if $(X, (W_n))$ is a numerical radius operator space, then there is a $W$-complete isometry $\Phi$ from $X$ to $\mathcal{B}(H)$ for some Hilbert space $H$ in the sense that $W_n(x)=w_n\left(\Phi_n(x)\right)$ for all $x\in \mathcal{M}_n(X)$ and $n\in \mathbb{N}$, where $w_n(\cdot )$ is the usual numerical radius norm on $\mathcal{B}(H^{(n)})$.

Having a look at the known equality
\[
\frac{1}{2}\|x\|=w\bigg(\begin{bmatrix}
0 & x\\ 0 & 0
\end{bmatrix}
\bigg), \quad  x\in \mathcal{B}(H).
\]
it is shown \cite{Itoh} that for a given numerical radius operator space $(X, (W_n))$ if one defines $\oo_n\,\,(n \in \mathbb{N})$ by
\begin{equation}\label{eq1.3}
\oo_n(x):=2 W_{2n}\bigg(\begin{bmatrix}
0 & x\\ 0 & 0
\end{bmatrix}\bigg),\quad x\in \mathcal{M}_n(X),
\end{equation}
then $X$ turns into an operator space. It is interesting to notice that if an operator space $(X, (O_n))$ is given, then there may be more than one operator space numerical radius $(W_n)$ satisfying \eqref{eq1.3}, \cite{Itoh}.
For instance, consider the maximal numerical radius norm $W_{\max}$  on an operator space $(X, (\oo_n))$, which is defined by
\[
W_{\max}(x)=\frac12 \inf \|aa^*+b^*b\|, \quad \mbox{for}\ x\in \mathcal{M}_n(X),
\]
where the infimum is taken over all decompositions $x=ayb$ with $\oo_r(y)=1$, $a\in \mathcal{M}_{n,r}(\mathbb{C})$, $y\in \mathcal{M}_r(X)$, $b\in \mathcal{M}_{r,n}(\mathbb{C})$,  $r\in \mathbb{N}$.
It is proved in \cite{Itoh} that $W_{\max}$ satisfies \eqref{eq1.1}, \eqref{eq1.2} and \eqref{eq1.3}.\\
There have been several generalizations of the usual numerical range in the last few decades. These concepts are useful in investigation of quantum error correction and perturbation theory  (e.g., see
\cite{CKZ, GR, LP, aa, aaa} and  references therein).
Several mathematicians \cite{Kit, HKS, KMY} established some interesting inequalities for the block matrix $\begin{bmatrix}
x & y\\ z & w
\end{bmatrix}$ and also its off-diagonal part, i.e. $\begin{bmatrix}
0 & y \\ z& 0
\end{bmatrix}$.
There are other papers involving numerical radius inequalities; cf. \cite{AK, SMY}.
In this paper, we obtain inequalities for $W_{2n}(\cdot)$ and $W_{\max}$ of $2\times 2$ block matrices with entries in appropriate matrix spaces similar to inequalities given in \cite{Kit}. These inequalities include bounds for $2\times 2$ block matrices. Furthermore, a generalization of a well known lemma given in \cite{Itoh} is established.

\section{Inequalities for operator space numerical radius and the maximal numerical radius norm }
In this section, we provide an inequality between operator space norm and operator space numerical radius similar to the usual operator norm and the usual numerical radius norm. Also we apply it to give bounds for the off-diagonal part $\begin{bmatrix}
0 & x\\ y & 0
\end{bmatrix}$ of the $2\times 2$ block matrix $\begin{bmatrix}
z & x\\ y & w
\end{bmatrix}$
defined on $\mathcal{M}_2(\mathcal{M}_n(X)).$
First we fix our notation and terminology.

Given abstract numerical radius operator spaces (resp., operator spaces) $X, Y$ and a
 linear map $\varphi$ from $X$ to $Y$, we define $\varphi_n$ from $\mathcal{M}_n(X)$ to $\mathcal{M}_n(Y)$ by
 \[
\varphi_n\left([x_{ij}]\right)=\big[\varphi (x_{ij})\big],\quad  [x_{ij}]\in \mathcal{M}_n(X).
\]
We denote the numerical radius norm (resp., the norm) of $x=[x_{ij}]\in \mathcal{M}_n(X)$ by $W_n(x)$ (resp., $\oo_n(x)$) and the norm of $\varphi_n$ by $W_n(\varphi_n)=\sup\{W_n\left(\varphi_n(x)\right)|x\in \mathcal{M}_n(X), W_n(x)\le 1\}$ (resp., $\oo_n(\varphi_n)=\sup\{\oo_n\left(\varphi_n(x)\right)|x\in \mathcal{M}_n(X), \oo_n(x)\le 1\}$.\\
The $W$-completely bounded norm (resp., completely bounded norm) of $\varphi$ is defined by
\[
W(\varphi)_{cb}=\sup\{W_n(\varphi_n)|n\in \mathbb{N}\}\quad
(\mbox{resp.,} \oo(\varphi)_{cb}=\sup\{\oo_n(\varphi_n)|n\in \mathbb{N}\}).
\]
We say $\varphi$ is $W$-completely bounded (resp., completely bounded) if $W(\varphi)_{cb}<\infty $ (resp., $\oo(\varphi)_{cb}<\infty$) and also we call $\varphi$ a $W$-complete isometry (resp., a complete isometry) if $W\left(\varphi_n(x)\right)=W_n(x)$ (resp., $\oo\left(\varphi_n(x)\right)=\oo_n(x)$) for each $x\in \mathcal{M}_n(X)$, $n\in \mathbb{N}$.

First of all we present a relation between $W_n(\cdot)$ and $\oo_n(\cdot)$.
\begin{lemma}\label{le2.1}
If $(X, (W_n))$ is an NROS, then there is an operator space norm $(\oo_n)$ on $X$ such that
\[
\frac{1}{2}\oo_n(x)\le W_n(x)\le \oo_n(x).
\]
for all $x\in \mathcal{M}_n(X)$ and $n\in \mathbb{N}$.
\end{lemma}
\begin{proof}
For given $(W_n(\cdot))$ and $x\in \mathcal{M}_n(X)$, we define $(\oo_n(\cdot))$ by $$\oo_n(x)=2W_{2n}\bigg(\begin{bmatrix}
0 & x\\ 0 & 0
\end{bmatrix}\bigg).$$
Then there exists a complete and $W$-complete isometry $\Phi$ from $X$ into $\mathcal{B}(H)$ \cite{Itoh}.
As $\Phi$ is a complete isometry, we have $\oo_n(x)=\|\Phi_n(x)\|_n$. In addition, since $\Phi$ is a $W$-complete isometry, we have $W_n(x)=w_n\left(\Phi_n(x)\right)$. Therefore,
\[
W_n(x)=w_n\left(\Phi_n(x)\right)\le \|\Phi_n(x)\|_n=\oo_n(x).
\]
and
\[
W_n(x)=w_n\left(\Phi_n(x)\right)\ge \frac{1}{2}\|\Phi_n(x)\|_n=\frac{1}{2}\oo_n(x).
\]
\end{proof}
The next result can be proved easily and we omit its proof.
\begin{lemma}
If $(X,(W_n))$ is an NROS and $U\in \mathcal{M}_n$ is a unitary, then
\begin{equation}\label{eq2.1}
W_n(U^*xU)=W_n(x)
\end{equation}
for any $x\in \mathcal{M}_n(X)$.
\end{lemma}
By a similar way,  identity \eqref{eq2.1} is valid for $W_{\max}$.
Also it should be mentioned here that $(\oo_n(\cdot))$ is unitarily invariant, i.e.
$\oo_n(UxV)= \oo_n(x)$ for all unitary $U, V\in \mathcal{M}_n$ and $x\in \mathcal{M}_n(X).$

Now, we use triangle inequality for $W_n(\cdot)$ and give upper and lower bounds for $W_{2n}\bigg(\begin{bmatrix}
0 & x\\ y & 0
\end{bmatrix}\bigg)$.

\begin{lemma}\label{le2.3}
If $(X,(W_n))$ is an NROS, then
\[\frac{1}{2}\max\left(\oo_n(x), \oo_n(y)\right)\leq W_{2n}\bigg(\begin{bmatrix}
0 & x\\ y & 0
\end{bmatrix}\bigg)\leq \frac12\left(\oo_n(x)+\oo_n(y)\right)
\]
for some operator space norm $(\oo_n(\cdot))$.
\end{lemma}
\begin{proof}
By \eqref{eq1.3}, there is an operator space norm $(\oo_n(\cdot))$ on $X$ such that
$$\oo_n(x)=2W_{2n}\bigg(\begin{bmatrix}
0 & x\\ 0 & 0
\end{bmatrix}\bigg).$$ First we prove the second inequality. Hence,
\begin{align*}
W_{2n}\bigg(\begin{bmatrix}
0 & x\\ y & 0
\end{bmatrix}\bigg)&\le W_{2n}\bigg(\begin{bmatrix}
0 & x\\ 0 & 0
\end{bmatrix}\bigg)+W_{2n}\bigg(\begin{bmatrix}
0 &0 \\y &0
\end{bmatrix}\bigg)\cr
&=\frac{1}{2}\oo_n(x)+W_{2n}\bigg(\begin{bmatrix}
0 & 1\\ 1 & 0
\end{bmatrix}\begin{bmatrix}
0 & y\\ 0 & 0
\end{bmatrix}\begin{bmatrix}
0 & 1 \\ 1 & 0
\end{bmatrix}\bigg)\cr
&\le \frac{1}{2}\oo_n(x)+W_{2n}\bigg(\begin{bmatrix}
0 & y \\ 0 & 0
\end{bmatrix}\bigg)\quad ({\rm by \ inequality~} \eqref{eq1.2})\cr
&=\frac{1}{2}\left(\oo_n(x)+\oo_n(y)\right).
\end{align*}
To proving the first inequality, we use Ruan's axioms as follows.
\begin{align*}
W_{2n}\bigg(\begin{bmatrix}
0 & x\\ y & 0
\end{bmatrix}\bigg)&\geq \frac12\oo_{2n}\bigg(\begin{bmatrix}
0 & x\\ y & 0
\end{bmatrix}\bigg) \quad ({\rm by \ Lemma~}\eqref{le2.1})\ \cr
&\ge\frac{1}{2}\oo_{2n}\bigg(\begin{bmatrix}
1 & 0\\ 0 & 0
\end{bmatrix}\begin{bmatrix}
0 & x\\ y & 0
\end{bmatrix}\begin{bmatrix}
0 & 0 \\ 1 & 0
\end{bmatrix}\bigg)\cr
&=\frac{1}{2}\oo_{2n}\bigg(\begin{bmatrix}
x & 0\\ 0 & 0
\end{bmatrix}\bigg)
=\frac{1}{2}\oo_n(x).
\end{align*}
Similarly $W_{2n}\bigg(\begin{bmatrix}
0 & x\\ y & 0
\end{bmatrix}\bigg)\ge \frac12\oo_n(y).$
\end{proof}

\begin{remark}\label{Re2.4}
Utilizing Lemma \ref{le2.1}, the inequalities of Lemma \ref{le2.3} can be stated as follows:
\[
\frac12\max\left(W_n(x), W_n(y)\right)\leq W_{2n}\bigg(\begin{bmatrix}
0 & x\\ y & 0
\end{bmatrix}\bigg)\leq W_n(x)+W_n(y).
\]
\end{remark}
Now we are in a position to verify a general inequality for $W_n(.),$ which contains some inequalities as special cases.

\begin{theorem}\label{Th2.5}
Let $(X, (W_n))$ be an NROS. Then for each $x,y\in \mathcal{M}_n(X)$ and
$\alpha,\beta,\gamma,\delta\in \mathcal{M}_n(\mathbb{C})$
\[
W_n(\alpha x\beta \pm \gamma y\delta)\le (\Vert\alpha\Vert \Vert\beta\Vert+ \Vert\gamma\Vert\Vert\delta\Vert)
\max(\oo_n(x), \oo_n(y)),
\]
where $(\oo_n(\cdot))$ is a certain operator space norm.
\end{theorem}
\begin{proof}
Assume that $(\oo_n(\cdot))$ is defined by \eqref{eq1.3}. Using the second inequality of Lemma \ref{le2.1}, Ruan's axioms of operator spaces and the $C^*$-identity, we have
\begin{align}
W_n(\alpha x\beta + \gamma y\delta)&\le \oo_n(\alpha x\beta + \gamma y\delta)=\oo_{2n}\bigg(\begin{bmatrix}
\alpha & \gamma
\end{bmatrix} \begin{bmatrix}
 x & 0\\ 0 & y
\end{bmatrix} \begin{bmatrix}
 \beta\\ \delta
\end{bmatrix}\bigg)\notag\\
& \le\Vert \begin{bmatrix}
\alpha & \gamma
\end{bmatrix}\Vert \oo_{2n}\bigg(\begin{bmatrix}
 x & 0\\ 0 & y
\end{bmatrix}\bigg)\Vert \begin{bmatrix}
 \beta\\ \delta
\end{bmatrix}\Vert \notag\\
& =\Vert \alpha\alpha^*+\gamma\gamma^*\Vert^{\frac12}\oo_{2n}\bigg(\begin{bmatrix}
 x & 0\\ 0 & y
\end{bmatrix}\bigg) \Vert \beta^*\beta+\delta^*\delta \Vert^{\frac12} \notag\\
 & \le \frac12\left(\Vert\alpha\alpha^*+\gamma\gamma^*\Vert+ \Vert \beta^*\beta+\delta^*\delta \Vert\right)
\oo_{2n}\bigg(\begin{bmatrix}
 x & 0 \\ 0 & y
\end{bmatrix} \bigg) \notag\\
& \le \frac12\left(\Vert\alpha\Vert^2+\Vert \beta\Vert^2+ \Vert \gamma\Vert^2+\Vert\delta \Vert^2\right)
\oo_{2n}\bigg(\begin{bmatrix}
 x & 0 \\ 0 & y
\end{bmatrix} \bigg) \label{ge2.2}
\end{align}
Let $t>0$. Replace $\alpha, \beta, \gamma, \delta$ by $t\alpha, t^{-1}\beta, t\gamma, t^{-1}\delta$, respectively, in inequality \eqref{ge2.2} and use the following equality
\[
\inf_{t>0} \frac{t^2u + t^{-2}v}{2}=\sqrt{uv}
\]
to get
\[
W_n(\alpha x\beta + \gamma y\delta)\le
\left(\Vert\alpha\Vert\Vert \beta\Vert+ \Vert \gamma\Vert\Vert\delta \Vert\right)
\max(\oo_n(x), \oo_n(y)).
\]
To completes the proof, it is sufficient to replace $y$ by $-y$ in the above inequality.
\end{proof}

\begin{corollary}
If $(X, (W_n))$ is an NROS, then there exists an operator space norm $(\oo_n(\cdot))$ such that for any $x,y\in \mathcal{M}_n(X)$ and
$\alpha,\beta\in \mathcal{M}_n(\mathbb{C})$, it holds that
\begin{eqnarray}\label{cor ge2.3}
W_n(\alpha x\beta \pm \beta y\alpha)\le 2\Vert\alpha\Vert \Vert\beta\Vert
\max(\oo_n(x), \oo_n(y)).
\end{eqnarray}
In particular,
\[
W_n(\alpha x\pm y\alpha)\le 2\Vert\alpha\Vert
\max(\oo_n(x), \oo_n(y)).
\]
and
\[
W_n(\alpha x\pm x\alpha)\le 2\Vert\alpha\Vert
\oo_n(x).
\]
\begin{proof}
To show inequality \eqref{cor ge2.3}, it is enough to take $\gamma=\beta$ and $\delta=\alpha$ in Theorem \ref{Th2.5}. The other inequalities follow immediately from inequality \eqref{cor ge2.3}.
\end{proof}
\end{corollary}
\begin{corollary}
Suppose $(X, (W_n))$ is an NROS. Then there exists an operator space norm $(\oo_n(\cdot))$ such that for any $x,y\in \mathcal{M}_n(X)$ and
$\alpha,\gamma\in \mathcal{M}_n(\mathbb{C})$, it holds that
\[
W_n(\alpha x \pm \gamma y)\le (\Vert\alpha\Vert + \Vert\gamma\Vert)
\max(\oo_n(x), \oo_n(y)).
\]
In particular,
\[
W_n(\alpha x \pm \gamma x)\le (\Vert\alpha\Vert + \Vert\gamma\Vert)
\oo_n(x).
\]
\end{corollary}
\begin{proof}
The first inequality immediately follows from taking $\beta=\delta=I$ in Theorem \ref{Th2.5}, and for the second inequality it is sufficient to put $x=y$ in the first inequality.
\end{proof}

Next we present more results for the operator space numerical radius  of $2\times 2$ off-diagonal block matrices.
To do this, we need the following lemma.
\begin{lemma}\label{le2.4}
Let $(X,(W_n))$ be an NROS. Then for each $x,y\in \mathcal{M}_n(X)$
\begin{enumerate}
\item[(a)]
$W_{2n}\bigg(\begin{bmatrix}
0 & x\\ e^{i\theta}y & 0
\end{bmatrix}\bigg)=W_{2n}\bigg(\begin{bmatrix}
0 & x\\ y& 0
\end{bmatrix}\bigg)$ $\mbox{for}~~\theta\in\mathbb{R},$
\item[(b)]
$W_{2n}\bigg(\begin{bmatrix}
0 & x \\ y & 0
\end{bmatrix}\bigg)=W_{2n}\bigg(\begin{bmatrix}
0 & y \\ x & 0
\end{bmatrix}\bigg),$
\item[(c)]
$W_{2n}\bigg(\begin{bmatrix}
x & y\\ y & x
\end{bmatrix}\bigg)=\max \left(W_n(x+y),W_n(x-y)\right),$\\
In particular,
\[
W_{2n}\bigg(\begin{bmatrix}
0 & y\\ y & 0
\end{bmatrix}\bigg)=W_n(y).
\]
\item[(d)]
$W_{2n}\bigg(\begin{bmatrix}
y & -x\\ x & y
\end{bmatrix}\bigg)= \max \left(W_n(x+iy),W_n(x-iy)\right).$
\end{enumerate}
Note that if $(X, (\oo_n))$ is an operator space, then all above statements hold for $W_{\max}$.
\end{lemma}
\begin{proof}
Parts (a) and (b) can be easily concluded by utilizing identity \eqref{eq2.1} to the matrix $\begin{bmatrix}
0 & x\\ y &0
\end{bmatrix}$ and the unitary operators $\begin{bmatrix}
I& 0\\0 & e^{\frac{i\theta}{2}} I
\end{bmatrix}$ and $\begin{bmatrix}
0 & I \\ I & 0
\end{bmatrix}$, respectively. Part (c) follows from applying identity \eqref{eq2.1} to the matrix $\begin{bmatrix}
x & y\\ y & x
\end{bmatrix}$ and the unitary $\frac{1}{\sqrt{2}}\begin{bmatrix}
I & I \\ -I & I
\end{bmatrix}$. To verify part (d), first we use identity \eqref{eq2.1} to the matrix
$\begin{bmatrix}
iy & -x\\ x & iy
\end{bmatrix}$ and the unitary $\frac{1}{\sqrt{2}}\begin{bmatrix}
I & iI \\ iI & I
\end{bmatrix}$ to get
\[
W_{2n}\bigg(\begin{bmatrix}
 iy & -x\\ x & iy
\end{bmatrix}\bigg)= \max \left(W_n(x+y),W_n(x-y)\right).
\]
Taking $-iy$ instead of $y$ in the above identity we reach part (d).
\end{proof}
Our first main result is stated as follows.
\begin{theorem}\label{Th2.9}
Let $(X,(W_n))$ be an NROS and $x,y\in \mathcal{M}_n(X)$. Then
\[
W_{2n}\bigg(\begin{bmatrix}
0 & x\\ y & 0
\end{bmatrix}\bigg)\ge \frac{1}{2}\max \left(W_n(x+y), W_n(x-y)\right)
\]
and
\[
W_{2n}\bigg(\begin{bmatrix}
0 & x \\ y & 0
\end{bmatrix}\bigg)\le \frac{1}{2}\left(W_n(x+y)+W_n(x-y)\right).
\]
\end{theorem}
\begin{proof}
\begin{align*}
W_n(x+y)&= W_n\bigg( \begin{bmatrix}
1 & 1
\end{bmatrix} \begin{bmatrix}
0 & x\\ y & 0
\end{bmatrix} \begin{bmatrix}
 1 \\ 1
\end{bmatrix}\bigg)\cr
&\le \left\Vert \begin{bmatrix}
1 & 1 \end{bmatrix}\right\Vert^2
W_{2n}\bigg(\begin{bmatrix}
0 & x\\ y & 0
\end{bmatrix}\bigg) \quad ({\rm by\ inequality~} \eqref{eq1.2})\cr
&= 2W_{2n} \bigg(\begin{bmatrix}
0 & x\\ y & 0
\end{bmatrix}\bigg).
\end{align*}
Hence,
\begin{equation}\label{eq2.2}
\frac{1}{2}W_n(x+y)\le W_{2n}\bigg(\begin{bmatrix}
0 & x\\ y & 0
\end{bmatrix}\bigg).
\end{equation}
Replacing $y$ by $-y$ in inequality \eqref{eq2.2}, we get
\begin{equation}\label{eq2.3}
\frac{1}{2}W_n(x-y)\le W_{2n}\bigg( \begin{bmatrix}
0 & x\\ -y & 0
\end{bmatrix}\bigg)
=W_{2n}\bigg(\begin{bmatrix}
0 & x \\ y & 0
\end{bmatrix}\bigg)~ \text{(by Lemma \ref{le2.4} (a))}
\end{equation}
Now, the first inequality follows from inequalities \eqref{eq2.2} and \eqref{eq2.3}.
To prove the second inequality, we apply triangle inequality and Lemma \ref{le2.4} as follows:
\begin{align*}
W_n(x+y)+W_n(x-y)&=
W_{2n}\bigg(\begin{bmatrix}
0 & x+y\\ x+y & 0
\end{bmatrix}\bigg)+ W_{2n} \bigg(\begin{bmatrix}
0 & x-y\\ x-y & 0
\end{bmatrix}\bigg)\cr
&=W_{2n}\bigg(\begin{bmatrix}
0 & x+y\\ x+y & 0
\end{bmatrix}\bigg)+ W_{2n} \bigg(\begin{bmatrix}
0 & x-y\\ y-x & 0
\end{bmatrix}\bigg)\cr
&\hspace{3.8cm} \text{(by Lemma \ref{le2.4} (a) and (c))}\cr
&\ge W_{2n} \bigg( \begin{bmatrix}
0 & x+y \\ x+y & 0
\end{bmatrix}+\begin{bmatrix}
0 & x-y \\ y-x & 0
\end{bmatrix}\bigg)\cr
&=2 W_{2n} \bigg(\begin{bmatrix}
0 & x\\ y & 0
\end{bmatrix}\bigg).
\end{align*}
\end{proof}
\begin{corollary}
If $(X,(W_n))$ is an NROS and $x,y\in \mathcal{M}_n(X)$, then
\[
\max \left(W_n(x), W_n(y)\right)\le
W_{2n}\bigg(\begin{bmatrix}
0 & x+y \\ x-y & 0
\end{bmatrix}\bigg)\le W_n(x)+W_n(y).
\]
\end{corollary}
\begin{proof}
It's enough to take $x+y$ and $x-y$ instead of $x$ and $y$, respectively, in Theorem \ref{Th2.9}.
\end{proof}
\begin{proposition}
Suppose $(X,(W_n))$ is an NROS and $x,y \in \mathcal{M}_n(X)$. Then
\[
W_{2n}\bigg( \begin{bmatrix}
0 & x\\ y & 0
\end{bmatrix}\bigg) \le \min\left(W_n(x),W_n(y)\right) +\frac{\min \left( \oo_n(x+y),\oo_n(x-y)\right)}{2}
\]
for some operator space norm $(\oo_n(\cdot))$.
\end{proposition}
\begin{proof}
By Lemma \ref{le2.4} (a), (b) and identity \eqref{eq1.3}, we get
\begin{align}\label{eq*}
\frac{1}{2}\oo_n(x+y)+W_n(y) &=
W_{2n}\bigg(\begin{bmatrix}
0 & x+y\\ 0 & 0
\end{bmatrix}\bigg)+ W_{2n} \bigg(\begin{bmatrix}
0 & y \\ y & 0
\end{bmatrix}\bigg)\cr
&=W_{2n}\bigg(\begin{bmatrix}
0 & x+y\\ 0 & 0
\end{bmatrix}\bigg)+ W_{2n} \bigg(\begin{bmatrix}
0 & -y \\ y & 0
\end{bmatrix}\bigg)\cr
&\ge W_{2n}\bigg(\begin{bmatrix}
0 & x\\ y & 0
\end{bmatrix}\bigg) (\rm{by\ triangle\ inequality})
\end{align}
Replacing $y$ by $-y$ in inequality \eqref{eq*} and using Lemma \ref{le2.4} (a), we obtain
\begin{equation}\label{eq&}
W_{2n} \bigg( \begin{bmatrix}
 0 & x\\ y & 0
\end{bmatrix}\bigg)\le \frac{1}{2}\oo_n(x-y) +W_n(y).
\end{equation}
It follows from inequalities \eqref{eq*} and \eqref{eq&} that
\begin{equation}\label{eq2.4}
W_{2n}\bigg(\begin{bmatrix}
0 & x \\ y & 0
\end{bmatrix}\bigg)\le \frac{\min \left( \oo_n(x+y),\oo_n(x-y)\right)}{2}+W_n(y).
\end{equation}
Interchanging $x$ and $y$ in inequality \eqref{eq2.4} and using Lemma \ref{le2.4} (b), we get
\begin{equation}\label{eq2.5}
W_{2n}\bigg(\begin{bmatrix}
0 & x \\ y & 0	
\end{bmatrix}\bigg)\le \frac{\min \left(\oo_n(x+y),\oo_n(x-y)\right)}{2}+W_n(x).
\end{equation}
Now the result follows from inequalities \eqref{eq2.4} and \eqref{eq2.5}.
\end{proof}
\begin{theorem}
Let $(X,(W_n))$ be an NROS and $x, y\in \mathcal{M}_n(X)$. Then
\[
W_{2n} \bigg( \begin{bmatrix}
 0 & x\\ y & 0
\end{bmatrix}\bigg)\ge \bigg\vert\frac{1}{2}\max\left(\oo_n(x+y), \oo_n(x-y)\right) -
\min\left(W_n(x), W_n(y)\right)\bigg\vert,
\]
and
\[
W_{2n} \bigg( \begin{bmatrix}
 0 & x\\ y & 0
\end{bmatrix}\bigg)\ge \bigg\vert \max\left(W_n(x), W_n(y)\right) - \frac{1}{2}\min\left(\oo_n(x+y), \oo_n(x-y)\right)
\bigg\vert.
\]
\end{theorem}
\begin{proof}
Utilizing identity \eqref{eq1.3}, Lemma \ref{le2.4} (a) and (c), we get
\begin{align} \label{eq2.6}
\frac12 \oo_n(x+y) &=W_{2n} \bigg( \begin{bmatrix}
 0 & x+y\\ 0 & 0
\end{bmatrix}\bigg) = W_{2n} \bigg( \begin{bmatrix}
 0 & x\\ y & 0
\end{bmatrix} + \begin{bmatrix}
 0 & y\\ -y & 0 \end{bmatrix}\bigg)\notag \\
& \le W_{2n} \bigg( \begin{bmatrix}
 0 & x\\ y & 0
\end{bmatrix}\bigg) +W_n(y).
\end{align}
Replacing $y$ by $-y$ in inequality (\ref{eq2.6}) and using Lemma \ref{le2.4} (a) we have
\begin{eqnarray}\label{eq2.7}
\frac12 \oo_n(x-y) \le W_{2n} \bigg( \begin{bmatrix}
 0 & x\\ y & 0
\end{bmatrix}\bigg) +W_n(y).
\end{eqnarray}
So, by inequalities \eqref{eq2.6} and \eqref{eq2.7}

\begin{eqnarray}\label{eq2.8}
\frac12 \max \left(\oo_n(x+y), \oo_n(x-y)\right) \le W_{2n} \bigg( \begin{bmatrix}
 0 & x\\ y & 0
\end{bmatrix}\bigg) +W_n(y).
\end{eqnarray}
Interchanging $x$ and $y$ in inequality \eqref{eq2.8} and using Lemma \ref{le2.4} (b) we reach
\begin{eqnarray}\label{eq2.9}
\frac12 \max \left(\oo_n(x+y), \oo_n(x-y)\right) \le W_{2n} \bigg( \begin{bmatrix}
 0 & x\\ y & 0
\end{bmatrix}\bigg) +W_n(x).
\end{eqnarray}
It follows from inequalities \eqref{eq2.8} and \eqref{eq2.9} that
\begin{equation}\label{eq2.10}
\frac12 \max \left(\oo_n(x+y), \oo_n(x-y)\right) -\min\left( W_n(x), W_n(y)\right)
\le W_{2n} \bigg( \begin{bmatrix}
 0 & x\\ y & 0
\end{bmatrix}\bigg).
\end{equation}
On the other hand, by identity \eqref{eq1.3}, we have
\begin{align}\label{eq2.11}
W_{2n} \bigg( \begin{bmatrix}
 0 & x\\ y & 0
\end{bmatrix}\bigg) +\frac12 \oo_n(x-y)&= W_{2n} \bigg( \begin{bmatrix}
 0 & x\\ y & 0
\end{bmatrix}\bigg) + W_{2n} \bigg( \begin{bmatrix}
 0 & x-y\\ 0 & 0
\end{bmatrix}\bigg)\notag\\
& \ge W_{2n} \bigg( \begin{bmatrix}
 0 & x\\ y & 0
\end{bmatrix} - \begin{bmatrix}
 0 & x-y\\ 0 & 0
\end{bmatrix}\bigg)\notag\\
&= W_{2n} \bigg( \begin{bmatrix}
 0 & y\\ y & 0
\end{bmatrix}\bigg)= W_n(y).
\end{align}
Again, by replacing $y$ by $-y$ in inequality \eqref{eq2.11} and using Lemma \ref{le2.4} (a), we get
\begin{eqnarray}\label{eq2.12}
W_n(y) \le W_{2n} \bigg( \begin{bmatrix}
 0 & x\\ y & 0
\end{bmatrix}\bigg) +\frac12 \oo_n(x+y).
\end{eqnarray}
We therefore infer, by inequalities \eqref{eq2.11} and \eqref{eq2.12}, that
\begin{eqnarray}\label{eq2.13}
W_n(y) \le W_{2n} \bigg( \begin{bmatrix}
 0 & x\\ y & 0
\end{bmatrix}\bigg) +\frac12 \max\left(\oo_n(x+y), \oo_n(x-y)\right).
\end{eqnarray}
In inequality\eqref{eq2.13} we interchange $x$ and $y$ and use Lemma \ref{le2.4} (b) to get
\begin{eqnarray}\label{eq2.14}
W_n(x) \le W_{2n} \bigg( \begin{bmatrix}
 0 & x\\ y & 0
\end{bmatrix}\bigg) +\frac12 \max\left(\oo_n(x+y), \oo_n(x-y)\right).
\end{eqnarray}
It follows from inequalities \eqref{eq2.13} and \eqref{eq2.14} that
\begin{equation}\label{eq2.15}
-\Big(\frac12 \max\left(\oo_n(x+y), \oo_n(x-y)\right) -\min\left(W_n(x), W_n(y)\right)\Big)
 \le W_{2n} \Big( \begin{bmatrix}
 0 & x\\ y & 0
\end{bmatrix}\Big).
\end{equation}
Thus the first desired inequality follows immediately from inequalities \eqref{eq2.10} and \eqref{eq2.15}.

The other inequality is deduced by a similar argument.
\end{proof}
In the sequel, we present some inequalities for $W_{\max}$ having common nature to our earlier results. The next theorem is one of our main results.
\begin{theorem}
Let $(X, (\oo_n))$ be an operator space. Then
\begin{eqnarray*}
\frac12\max\left(W_{\max}(x_1+x_2), W_{\max}(x_1-x_2) \right) &\le& W_{\max}\Big(\begin{bmatrix} 0 & x_1 \\ x_2 & 0 \end{bmatrix}\Big)\\
&\le&
\frac12\left(W_{\max}(x_1+x_2)+ W_{\max}(x_1-x_2) \right)
\end{eqnarray*}
for all $x_1, x_2\in \mathcal{M}_n(X)$.
\end{theorem}
\begin{proof}
For the first inequality, let $\begin{bmatrix} 0 & x_1 \\ x_2 & 0 \end{bmatrix}=ayb$, $\oo_r(y)=1,$ for $a\in \mathcal{M}_{n,r}(\mathbb{C})$, $y\in \mathcal{M}_r(X)$, $b\in \mathcal{M}_{r,n}(\mathbb{C})$ and  $r\in \mathbb{N}$. So, we can write
\[
x_1+x_2=\begin{bmatrix}1 & 1 \end{bmatrix}
\begin{bmatrix} 0 & x_1 \\ x_2 & 0 \end{bmatrix}
\begin{bmatrix}1 \\ 1\end{bmatrix} =
\begin{bmatrix}1 & 1\end{bmatrix}
 ayb \begin{bmatrix}1 \\ 1\end{bmatrix}.
\]
We derive from the definition of $W_{\max}(x_1+x_2)$ that
\begin{align*}
\frac12 W_{\max}(x_1+x_2)&\le \frac14 \Big\Vert \begin{bmatrix}1 & 1\end{bmatrix} aa^* \begin{bmatrix}1 \\ 1\end{bmatrix}
+ \begin{bmatrix}1 & 1\end{bmatrix} b^*b \begin{bmatrix}1 \\ 1\end{bmatrix}
\Big \Vert\\
& =\frac14 \Big\Vert \begin{bmatrix}1 & 1\end{bmatrix} (aa^*+b^*b)
 \begin{bmatrix}1 \\ 1\end{bmatrix}
\Big \Vert\\
& \le \frac12 \Vert aa^*+b^*b \Vert.
\end{align*}
whence
\begin{equation}\label{eq2.20}
\frac12 W_{\max}(x_1+x_2)\le W_{\max}\Big(\begin{bmatrix} 0 & x_1 \\ x_2 & 0 \end{bmatrix}\Big).
\end{equation}
Replacing $x_2$ by $-x_2$ in inequality \eqref{eq2.20} and using Lemma \ref{le2.4} (a) for $W_{\max}$, we get
\begin{equation}\label{eq2.21}
\frac12 W_{\max}(x_1- x_2)\le W_{\max}\Big(\begin{bmatrix} 0 & x_1 \\ x_2 & 0 \end{bmatrix}\Big).
\end{equation}
The first inequality now deduce from inequalities \eqref{eq2.20} and \eqref{eq2.21}.\\
For the second inequality, it's sufficient to prove that
\[
W_{\max}\Big(\begin{bmatrix} 0 & x_1+x_2 \\ x_1-x_2 & 0 \end{bmatrix}\Big)\le
W_{\max}(x_1)+ W_{\max}(x_2).
\]
For any $x_1,x_2 \in \mathcal{M}_n(X) $ and given  $\epsilon>0,$ we may choose  $a_i\in \mathcal{M}_{n,r}(\mathbb{C})$, $b_i\in \mathcal{M}_{r,n}(\mathbb{C})$, $y_i\in \mathcal{M}_r(X)$ with $\oo_r(y_i)=1$ such that $x_i=a_i y_i b_i\,\,(i=1,2)$  and
\[
W_{\max}(x_1) +\epsilon \ge \frac12 \Vert a_1a_1^*+b_1^*b_1\Vert,~~\quad
W_{\max}(x_2) +\epsilon \ge \frac12 \Vert a_2a_2^*+b_2^*b_2\Vert.
\]
Now we can write the following representation:
\[
\begin{bmatrix}0 & x_1+x_2 \\ x_1-x_2 & 0 \end{bmatrix}=
\begin{bmatrix}a_1 & a_2&0&0\\ 0&0&a_1&a_2\end{bmatrix}
\begin{bmatrix}y_1&0&0&0 \\ 0&y_2&0 & 0\\0&0&y_1&0\\0&0&0&y_2\end{bmatrix}
\begin{bmatrix}0 &b_1\\0&b_2\\b_1&0\\-b_2&0\end{bmatrix}.
\]
It follows that
\begin{align*}
&\hspace{-0.5cm}W_{\max}\Big(\begin{bmatrix} 0 & x_1+x_2 \\ x_1-x_2 & 0 \end{bmatrix}\Big)\\&\le
\frac12\Bigg\Vert
\begin{bmatrix}a_1 & a_2&0&0\\ 0&0&a_1&a_2\end{bmatrix}
\begin{bmatrix}a_1^* &0\\a_2^*&0\\0&a_1^*\\0&a_2^*\end{bmatrix}+
\begin{bmatrix}0 & 0&b_1^*&-b_2^*\\ b_1^*&b_2^*&0&0 \end{bmatrix}
\begin{bmatrix}0 &b_1\\0&b_2\\b_1&0\\-b_2&0\end{bmatrix}  \Bigg\Vert\\
 & =\frac12\big\Vert a_1a_1^*+a_2a_2^*+b_1^*b_1+b_2^*b_2 \big\Vert\\
 &\le \frac12\big\Vert a_1a_1^*+b_1^*b_1\big\Vert+\frac12 \big\Vert a_2a_2^*+b_2^*b_2 \big\Vert\\
 & \le W_{\max}(x_1) + W_{\max}(x_2) +2\epsilon.
\end{align*}
Letting $\epsilon\rightarrow 0,$ we get the required inequality.
\end{proof}
 In the next result, other lower and upper bounds for $W_{\max}$ are furnished.
\begin{proposition}
Suppose $(X, (\oo_n))$ is an operator space. Then
\[
\frac12\max\left(W_{\max}(x_1), W_{\max}(x_2) \right) \le
W_{\max}\Big(\begin{bmatrix} 0 & x_1 \\ x_2 & 0 \end{bmatrix}\Big)\le
W_{\max}(x_1)+ W_{\max}(x_2)
\]
for $x_1, x_2\in \mathcal{M}_n(X)$.
\end{proposition}
\begin{proof}
It turns out from inequalities \eqref{eq2.20} and \eqref{eq2.21} that
\begin{align*}
2W_{\max}\Big(\begin{bmatrix} 0 & x_1 \\ x_2 & 0 \end{bmatrix}\Big)&\ge
\frac12 W_{\max}(x_1+x_2)+ \frac12 W_{\max}(x_1-x_2)\\
&\ge \frac12  W_{\max}(x_1+x_2+x_1-x_2)=W_{\max}(x_1).
\end{align*}
Therefore,
\begin{eqnarray}\label{eq2.22}
W_{\max}\Big(\begin{bmatrix} 0 & x_1 \\ x_2 & 0 \end{bmatrix}\Big)\ge
\frac12 W_{\max}(x_1).
\end{eqnarray}
In a similar manner,
\begin{eqnarray}\label{eq2.23}
W_{\max}\Big(\begin{bmatrix} 0 & x_1 \\ x_2 & 0 \end{bmatrix}\Big)\ge
\frac12 W_{\max}(x_2).
\end{eqnarray}
The first inequality follows immediately from \eqref{eq2.22} and \eqref{eq2.23}.
To get the second inequality assume $x_1,x_2 \in \mathcal{M}_n(X) $
and   $\epsilon>0.$ we may select  $a_i\in \mathcal{M}_{n,r}(\mathbb{C})$, $b_i\in \mathcal{M}_{r,n}(\mathbb{C})$, $y_i\in \mathcal{M}_r(X)$ with $x_i=a_i y_i b_i\,\,(i=1,2)$  and
\[
W_{\max}(x_1) +\epsilon \ge \frac12 \Vert a_1a_1^*+b_1^*b_1\Vert,\quad
W_{\max}(x_2) +\epsilon \ge \frac12 \Vert a_2a_2^*+b_2^*b_2\Vert.
\]
The decomposition
\[
\begin{bmatrix} 0 & x_1 \\ x_2 & 0 \end{bmatrix}=
\begin{bmatrix} a_1 & 0 \\ 0 & a_2 \end{bmatrix}
\begin{bmatrix} y_1 & 0 \\ 0 & y_2 \end{bmatrix}
\begin{bmatrix} 0 & b_1 \\ b_2 & 0 \end{bmatrix}
\]
yields that
\begin{align}
W_{\max}\Big(\begin{bmatrix} 0 & x_1 \\ x_2 & 0 \end{bmatrix}\Big)&\le
\frac12 \Big\Vert \begin{bmatrix} a_1 & 0 \\ 0 & a_2 \end{bmatrix}
\begin{bmatrix} a_1 & 0 \\ 0 & a_2 \end{bmatrix}^*+
\begin{bmatrix} 0 & b_1 \\ b_2 & 0 \end{bmatrix}^*
\begin{bmatrix} 0 & b_1 \\ b_2 & 0 \end{bmatrix} \Big\Vert \notag\\
& =\frac12 \Big\Vert \begin{bmatrix} a_1a_1^*+b_2^*b_2 & 0 \\ 0 & a_2a_2^*+b_1^*b_1 \end{bmatrix}\Big\Vert \notag\\
& =\frac12 \max\left(\Vert a_1a_1^*+b_2^*b_2 \Vert, \Vert a_2a_2^*+b_1^*b_1 \Vert\right) \notag\\
&\le\frac12\Vert a_1a_1^*+b_2^*b_2+ a_2a_2^*+b_1^*b_1\Vert \label{eq2.24}\\
& \le \frac12 \Vert a_1a_1^*+b_1^*b_1\Vert +\frac12\Vert a_2a_2^*+b_2^*b_2 \Vert \notag\\
& \le W_{\max}(x_1) + W_{\max}(x_2) +2\epsilon.\notag
\end{align}
where inequality \eqref{eq2.24} follows from the fact that, if $A,B\in \mathcal{B}(H) $ are positive operator, then $\max(\Vert A\Vert, \Vert B\Vert)\le \Vert A+ B\Vert$. Now since $\epsilon>0$ is arbitrary, we obtain the desired inequality.
\end{proof}
\section{ Upper and Lower Bounds of $2\times 2$ block matrices}
In this section, first we present some pinching inequalities for $W_n$. Moreover, we provide different bounds for
$2\times 2$ block matrices of the form $\begin{bmatrix}
x & y\\ z & w
\end{bmatrix}.$ Some other related inequalities are also discussed.

\begin{lemma}\label{le3.1}
Assume $(X, (W_n))$ is an NROS and $x,y,z,w\in \mathcal{M}_n(X)$.  Then
\[
W_{2n}\bigg(\begin{bmatrix}
x & 0 \\ 0 & w
\end{bmatrix}\bigg)
\le W_{2n}\bigg(\begin{bmatrix}
x & y \\ z & w
\end{bmatrix}\bigg),
\]
$\hspace{1cm}$and
\[
W_{2n}\bigg(\begin{bmatrix}
0 & y \\ z & 0
\end{bmatrix}\bigg)\le W_{2n}\bigg(\begin{bmatrix}
x & y \\ z & w
\end{bmatrix}\bigg).
\]
\begin{proof}
The first inequality can easily follows from $A=\begin{bmatrix}
x& y\\ z& w
\end{bmatrix}$, by considering unitary $U=\begin{bmatrix}
I & 0 \\ 0 & -I
\end{bmatrix}$, triangle inequality and identity \eqref{eq2.1} as
\[
\begin{bmatrix}
x &0\\ 0 & w
\end{bmatrix}=\frac{A+U^*AU}{2}.
\]
For the second inequality, we use
\[
\begin{bmatrix}
0 & y\\ z & 0
\end{bmatrix}=\frac{A-U^*AU}{2}.
\]
\end{proof}
\end{lemma}
\begin{proposition}
Let $(X, (W_n))$ be an NROS and $x,y \in \mathcal{M}_n(X)$. Then
\begin{equation}\label{equ3.1}
\max \left(W_n (x), W_n(y)\right) \le W_{2n} \bigg( \begin{bmatrix}
x & y \\ -y & -x
\end{bmatrix}
\bigg) \le W_n (x) + W_n (y).
\end{equation}
\end{proposition}
\begin{proof}
On making use of Lemma \ref{le3.1}, we get
\[
W_n(x) = W_{2n} \bigg( \begin{bmatrix}
x & 0 \\ 0 & -x
\end{bmatrix} \bigg)
\le
W_{2 n} \bigg(
\begin{bmatrix}
x & y \\ -y & -x
\end{bmatrix} \bigg)
\]
$\hspace{.5cm}$ and
\[
W_n(y) = W_{2n} \bigg( \begin{bmatrix}
0 & y \\ -y & 0
\end{bmatrix} \bigg)
\le
W_{2n} \bigg(
\begin{bmatrix}
x & y \\ -y & -x
\end{bmatrix} \bigg).
\]
Therefore,
\[
\max \left( W_n(x), W_n (y) \right) \le
W_{2n} \bigg(
\begin{bmatrix}
x & y \\ -y & -x
\end{bmatrix} \bigg).
\]
On the other hand, by employing triangle inequality, inequality \eqref{eq1.1}, Lemma \ref{le2.4} (a) and (c), we have
\begin{align*}
W_{ 2n} \bigg( \begin{bmatrix}
x & y \\ -y & -x
\end{bmatrix} \bigg)
&\le W_{2n} \bigg(
\begin{bmatrix}
x & 0 \\ 0 & -x
\end{bmatrix} \bigg) + W_{2n} \bigg(
\begin{bmatrix}
0 & y \\ -y & 0
\end{bmatrix} \bigg)\\
& = W_n(x) + W_n (y).
\end{align*}
\end{proof}

\begin{remark}
If we choose $y=x$ in inequality \eqref{equ3.1}, then for $x\in \mathcal{M}_n(X)$
\begin{equation}\label{equ3.2}
W_n (x) \le W_{2n} \bigg( \begin{bmatrix}
x & x \\ -x & -x
\end{bmatrix}
\bigg) \le 2W_n (x).
\end{equation}
 Now we show that
 \begin{equation}\nonumber
 W_{2n}\bigg(\begin{bmatrix}
x & x\\ -x & -x
\end{bmatrix}\bigg)=\oo_n(x),\quad x\in \mathcal{M}_n(X).
 \end{equation}
Using identities \eqref{eq2.1}, \eqref{eq1.3} with the unitary $U=\frac{1}{\sqrt{2}}\begin{bmatrix}
I & I \\ -I & I
\end{bmatrix}$ we have
\begin{align*}
W_{2n}\bigg( \begin{bmatrix}
x & x\\ -x & -x \end{bmatrix}\bigg)&=
W_{2n}\bigg(\frac12 \begin{bmatrix}
I & -I\\ I & I
\end{bmatrix} \begin{bmatrix}
x & x\\ -x & -x
\end{bmatrix}\begin{bmatrix}
I & I\\ -I & I
\end{bmatrix}\bigg)\\
&=\frac12 W_{2n}\bigg( \begin{bmatrix}
0 & 4x\\ 0 & 0 \end{bmatrix}\bigg)=2W_{2n}\bigg( \begin{bmatrix}
0 & x\\ 0 & 0 \end{bmatrix}\bigg)=\oo_n(x).
\end{align*}
Based on the above identity, one can conclude that the inequalities of Lemma \ref{le2.1} and inequalities \eqref{equ3.2} are equivalent.
\end{remark}
The next result provide a lower and upper bound for
$ \begin{bmatrix} x & y \\ z & w \end{bmatrix}.$
\begin{proposition}
Let $(X, (W_n))$ be an NROS and $x,y,z,w\in \mathcal{M}_n(X)$. Then
\[
W_{2n}\bigg(\begin{bmatrix}
x & y \\ z & w
\end{bmatrix}\bigg)\ge \max\left(W_n(x), W_n(w), \frac{W_n(y)}{2}, \frac{W_n(z)}{2}\right)
\]
and
\[
W_{2n}\bigg(\begin{bmatrix}
x & y \\ z & w
\end{bmatrix}\bigg)\le W_n(x)+W_n(y)+W_n(z)+W_n(w).
\]
\begin{proof}
Utilizing Lemma \ref{le3.1} and the first inequality of Remark \ref{Re2.4}, we derive
\begin{align*}
W_{2n}\bigg(\begin{bmatrix}
x & y \\ z & w
\end{bmatrix}\bigg)&\ge \max\bigg(W_{2n}\Big(\begin{bmatrix}
x & 0 \\ 0 & w
\end{bmatrix}\Big), W_{2n}\Big(\begin{bmatrix}
0 & y \\ z & 0
\end{bmatrix}\Big)\bigg)\\
& \ge \max\bigg(\max\left(W_n(x), W_n(w)\right), \max\left(\frac{W_n(y)}{2}, \frac{W_n(z)}{2}\right) \bigg)\\
& = \max\left(W_n(x), W_n(w), \frac{W_n(y)}{2}, \frac{W_n(z)}{2}\right).
\end{align*}
To verify the other inequality first we present an upper bound to the matrix
$\begin{bmatrix} x & y \\ 0 & 0 \end{bmatrix}$.
 To achieve this, we use the triangle inequality as follows:
\begin{align}
W_{2n}\bigg(\begin{bmatrix} x & y \\ 0 & 0 \end{bmatrix}\bigg)&\le
W_{2n}\bigg(\begin{bmatrix} x & 0 \\ 0 & 0 \end{bmatrix}\bigg)+
W_{2n}\bigg(\begin{bmatrix} 0 & y \\ 0 & 0 \end{bmatrix}\bigg)\nonumber\\
&=W_n(x) + \frac12 \oo_n(y) \notag\\
& \hspace{2.8cm}(\text{by inequality \eqref{eq1.1} and identity \eqref{eq1.3}})\notag\\
&\le W_n(x) + W_n(y) \quad\quad\quad (\text{by Lemma \ref{le2.1} })\label{eq3.1}
\end{align}
For the general case consider unitary  $\begin{bmatrix} 0 & I \\ I & 0 \end{bmatrix}.$
We infer by identity \eqref{eq2.1} that
\begin{align*}
W_{2n}\bigg(\begin{bmatrix} x & y \\ z & w \end{bmatrix}\bigg)&\le
W_{2n}\bigg(\begin{bmatrix} x & y \\ 0 & 0 \end{bmatrix}\bigg)+
W_{2n}\bigg(\begin{bmatrix} 0 & 0 \\ z & w \end{bmatrix}\bigg)\\
& = W_{2n}\bigg(\begin{bmatrix} x & y \\ 0 & 0 \end{bmatrix}\bigg)+
W_{2n}\bigg(U^*\begin{bmatrix} w & z \\ 0 & 0 \end{bmatrix}U\bigg )\\
& = W_{2n}\bigg(\begin{bmatrix} x & y \\ 0 & 0 \end{bmatrix}\bigg)+
W_{2n}\bigg(\begin{bmatrix} w & z \\ 0 & 0 \end{bmatrix}\bigg )\\
& \le W_n(x)+W_n(y)+W_n(z)+W_n(w).\\
&\hspace{5cm} (\text{by inequality \eqref{eq3.1}})
\end{align*}
\end{proof}
\end{proposition}
Another upper bound for $\begin{bmatrix} x & y \\ z & w \end{bmatrix}$ can be stated as follows.
\begin{theorem}
Let $(X, (W_n))$ be an NROS and $x,y,z,w\in \mathcal{M}_n(X)$. Then
\begin{align*}
W_{2n}\bigg(\begin{bmatrix}
x & y \\ z & w
\end{bmatrix}\bigg)\le \max\bigg(&\frac{W_n(x+w+i(y-z))}{2}, \frac{W_n(x+w-i(y-z))}{2}\bigg)\\
& + \frac{W_n(w-x)+W_n(y+z)}{2}.
\end{align*}
\begin{proof}
Applying identity \eqref{eq2.1} to the matrix $\begin{bmatrix} x & y \\ z & w \end{bmatrix}$ and unitary $U=\frac{1}{\sqrt{2}}\begin{bmatrix} I & -I \\ I & I \end{bmatrix}$, we have
\begin{align}\label{eq3.2}
W_{2n}\bigg(\begin{bmatrix}x & y \\ z & w\end{bmatrix}\bigg)&= W_{2n}\bigg(U^*\begin{bmatrix}x & y \\ z & w\end{bmatrix}U\bigg)\notag\\
& = \frac12 W_{2n}\bigg(\begin{bmatrix}x+y+z+w & -x+y-z+w \\ -x-y+z+w & x-y-z+w\end{bmatrix}\bigg)\\
& = \frac12 W_{2n}\bigg(\begin{bmatrix}x+w & y-z \\ z-y & x+w\end{bmatrix} +
\begin{bmatrix}y+z & w-x \\ w-x & -z-y \end{bmatrix}\bigg)\notag\\
& \le \frac12\bigg( W_{2n}\Big(\begin{bmatrix}x+w & y-z \\ z-y & x+w\end{bmatrix}\Big) +
W_{2n}\Big(\begin{bmatrix}y+z & w-x \\ w-x & -z-y \end{bmatrix}\Big)\bigg)\notag\\
& \le \frac12\bigg(\max\left(W_n(x+w+i(y-z)), W_n(x+w-i(y-z))\right)\notag\\
& + W_n(w-x)+W_n(y+z)\bigg).\quad
(\text{by Lemma \ref{le2.4} (c), (d)})\notag
\end{align}
\end{proof}
\end{theorem}
\begin{remark}
Suppose $(X, (W_n))$ is an NROS and $x,y,z,w\in \mathcal{M}_n(X)$. Then
\[W_{2n}\bigg(\begin{bmatrix}x & y \\ z & w\end{bmatrix}\bigg) \le
\max\left(W_n(x), W_n(w)\right)+ \frac{W_n(y+z)+W_n(y-z)}{2}.
\]
\begin{proof}
According to identity \eqref{eq3.2}, we can write
\begin{align*}
W_{2n}\bigg(\begin{bmatrix}x & y \\ z & w\end{bmatrix}\bigg)&=
\frac12 W_{2n}\bigg(\begin{bmatrix}x+w & w-x \\ w-x & x+w\end{bmatrix} +
\begin{bmatrix}y+z & y-z \\ z-y & -z-y \end{bmatrix}\bigg)\notag\\
& \le \frac12\bigg( W_{2n}\Big(\begin{bmatrix}x+w & w-x \\ w-x & x+w\end{bmatrix}\Big) +
W_{2n}\Big(\begin{bmatrix}y+z & y-z \\ z-y & -z-y \end{bmatrix}\Big)\bigg)\notag\\
& \le \max\left(W_n(x), W_n(w)\right)+ \frac{W_n(y+z)+W_n(y-z)}{2}.\\
& \hspace{5.5cm} (\text{by Lemma \ref{le2.4} (c)})
\end{align*}
\end{proof}
\end{remark}
The last result in this section is a generalization of a well known Lemma in \cite{Itoh}.
\begin{proposition}\label{le3.3}
Suppose $(X, (W_n))$ be an NROS. If $f \in \mathcal{M}_n(X)^*$ and $W^* (f) \le 1$, then there exists a state $P_0$ on $\mathcal{M}_n(\mathbb{C})$ such that
\begin{equation*}
\big| f( \alpha x \beta^* \pm \beta y \alpha^* ) \big| \le 2 P_0 ( \alpha \alpha^*)^{\frac{1}{2}} P_0 ( \beta \beta^*)^{\frac{1}{2}} W_{2n} \bigg( \begin{bmatrix}
0 & x \\ y & 0
\end{bmatrix} \bigg),
\end{equation*}
for all $\alpha,\beta \in \mathcal{M}_{n,r} ( \mathbb{C} )$, $x, y \in \mathcal{M}_r (X)$, $r\in\mathbb{N}$.\\
In addition,
\begin{equation}\label{ge3.5}
| f( \alpha x \beta \pm \gamma y \delta ) |\le \left( P_0 ( \alpha \alpha^*)^{\frac{1}{2}} P_0 ( \beta^* \beta)^{\frac{1}{2}}+
P_0 ( \gamma \gamma^*)^{\frac{1}{2}} P_0 ( \delta^* \delta)^{\frac{1}{2}}\right)
 \oo_{2n} \left( \begin{bmatrix}
0 & x \\ y & 0
\end{bmatrix} \right)
\end{equation}
for all $\alpha,\gamma \in \mathcal{M}_{n,r} ( \mathbb{C} )$, $x, y \in \mathcal{M}_r (X)$, $\beta,\delta \in \mathcal{M}_{r,n} ( \mathbb{C})$, $r\in\mathbb{N}$, where $W^*(f)= \sup\lbrace \vert f(x)\vert: x\in \mathcal{M}_n(X), W_n(x)\le 1 \rbrace$.
\end{proposition}

\begin{proof}
It is proved \cite{Itoh} under the same hypothesis that
\begin{align}
\bigg| f( \alpha x \alpha^*) \bigg| \le P_0 (\alpha \alpha^*) W_n (x) \label{eq3.3}\\
\bigg| f( \alpha x \beta ) \bigg| \le 2 P_0 ( \alpha \alpha^*)^{\frac{1}{2}} P_0 ( \beta^* \beta)^{\frac{1}{2}} W_{2n}\bigg( \begin{bmatrix}
0 & x \\ 0 & 0
\end{bmatrix}\bigg)\label{eq3.4}
\end{align}
Now by inequality \eqref{eq3.3}, we derive
\begin{align*}
\big| f ( \alpha x \beta^* + \beta y \alpha^* ) \big| & =\bigg| f \left(
\begin{bmatrix}
\alpha & \beta
\end{bmatrix} \begin{bmatrix}
0 & x \\ y & 0
\end{bmatrix}
\begin{bmatrix}
\alpha & \beta
\end{bmatrix}^* \right)\bigg|\\
& \le P_0( \alpha \alpha^* + \beta \beta^*) W_{2n}\bigg( \begin{bmatrix}
0 & x \\ y & 0
\end{bmatrix} \bigg).
\end{align*}
Let $t > 0$ and replace $\alpha$ and $\beta$ by $t\alpha$ and $\frac{1}{t} \beta$, respectively. Then the equality
\begin{equation}\label{eq3.5}
\inf_{t >0} \big\{ t^2 P_0 ( \alpha \alpha^*) + t^{-2} P_0 ( \beta \beta^* ) \big\} = 2 P_0 ( \alpha \alpha^*) ^{\frac{1}{2}} P_0( \beta \beta^* )^{\frac{1}{2}}
\end{equation}
ensures
\[
\big| f ( \alpha x \beta^* + \beta y \alpha^* ) \big|\le 2 P_0 ( \alpha \alpha^*)^{\frac{1}{2}} P_0 ( \beta \beta^*)^{\frac{1}{2}} W_{2n} \bigg( \begin{bmatrix}
0 & x \\ y & 0
\end{bmatrix} \bigg).
\]
Replace $y$ by $-y$ in the above inequality and use Lemma \ref{le2.4} (a) to deduce the first inequality of the proposition.

To verify inequality \eqref{ge3.5}, we apply inequality \eqref{eq3.4} as follows:
\begin{align*}
\big| f( \alpha x \beta + \gamma y \delta ) \big| &= \big| f\left(\begin{bmatrix}
\alpha & \gamma \end{bmatrix} \begin{bmatrix} 0 & x \\ y & 0
\end{bmatrix} \begin{bmatrix} \delta \\ \beta
\end{bmatrix} \right) \big|\\
& \le P_0( \begin{bmatrix}
\alpha & \gamma \end{bmatrix} \begin{bmatrix}
\alpha & \gamma \end{bmatrix}^*)^{\frac12} P_0(\begin{bmatrix} \delta \\ \beta
\end{bmatrix}^* \begin{bmatrix} \delta \\ \beta\end{bmatrix} )^{\frac12}\oo_{2n} \Big(\begin{bmatrix} 0 & x \\ y & 0 \end{bmatrix}\Big)\\
& \hspace{3cm} (\text{by inequality \eqref{eq3.4} and identity \eqref{eq1.3}})\\
& = P_0(\alpha \alpha^*+\gamma\gamma^*)^{\frac12} P_0(\beta^*\beta+ \delta^*\delta)^{\frac12}
\oo_{2n} \Big(\begin{bmatrix} 0 & x \\ y & 0 \end{bmatrix}\Big)\\
&\le \frac12 \left(P_0(\alpha \alpha^*+\gamma\gamma^*)+ P_0(\beta^*\beta+ \delta^*\delta)\right) \oo_{2n}\Big(\begin{bmatrix} 0 & x \\ y & 0 \end{bmatrix}\Big).\\
& \hspace{2.7cm} (\text{ by the arithmetic-geometric mean inequality})
\end{align*}
If we replace $\alpha, \beta, \gamma, \delta$ by $t\alpha, t^{-1}\beta, t\gamma, t^{-1}\delta$, respectively, in the above inequality, then from equality \eqref{eq3.5} we get
\begin{equation}\label{eq3.6}
\big| f( \alpha x \beta + \gamma y \delta ) \big| \le \left( P_0 ( \alpha \alpha^*)^{\frac{1}{2}} P_0 ( \beta^* \beta)^{\frac{1}{2}}+
P_0 ( \gamma \gamma^*)^{\frac{1}{2}} P_0 ( \delta^* \delta)^{\frac{1}{2}}\right)
 \oo_{2n} \Big( \begin{bmatrix}
0 & x \\ y & 0
\end{bmatrix} \Big).
 \end{equation}
 Taking $-y$ instead of $y$ in inequality \eqref{eq3.6} and using Lemma \ref{le2.4} (a), we reach inequality \eqref{ge3.5}.
\end{proof}
Noting that by letting $y=0$, $\gamma=\delta=0$ in inequality \eqref{ge3.5} and applying the first inequality of Lemma \ref{le2.1}, we obtain inequality \eqref{eq3.4}.

\bibliographystyle{amsplain}

\end{document}